\documentclass[12pt]{article}
\usepackage[dvipsnames,table,xcdraw]{xcolor}
\usepackage{url,floatflt}
\usepackage{helvet,times}
\usepackage{graphics}
\usepackage{mathptmx,amsmath,amssymb,bm}
\usepackage{amsthm}
\usepackage{float}
\usepackage[bf,hypcap]{caption}

\usepackage[T1]{fontenc}
\usepackage[polish,english]{babel}
\usepackage[utf8]{inputenc}

\usepackage[left=2cm,right=1.5cm,top=1.5cm,bottom=1.5cm,bindingoffset=0cm]{geometry}

\def\keywords{\vspace{.5em}
	{\textit{Keywords}:\,\relax%
}}

\newtheorem{theorem}{Theorem}[section]

\newtheorem{lemma}[theorem]{Lemma}
\newtheorem{remark}[theorem]{Remark}
\newtheorem{proposition}[theorem]{Proposition}

\begin{document}

\title{On the dimension of closed subspaces embedded into Banach spaces }

\author{Alexander A. Balinsky {$^1$} and Yarema A. Prykarpatskyy{$^2$} {$^3$}\\ 
$^1$Department of Mathematics,
	Cardiff University, Cardiff CF24 4AG, UK\\
	$^2$Department of Applied Mathematics, University of Agriculture in Kraków,\\
	 Aleja Mickiewicza 21, Kraków, 31-120, Poland \\
	$^3$Institute of Mathematics of NAS of Ukraine, 3 Tereschenkivska st., Kyiv-4, 01024, Ukraine
}

\begin{abstract}
Let $\mu$ be a probability measure on a measurable space $M$ and
$S_{p}^{(q)}$ be a linear subspace of $L_{p}(M,\,d\mu)$, $p>1$,
which is closed in $L_{p}(M,\,d\mu)$ and contained in
$L_{q}(M,\,d\mu)$ for some $q\geq 2$ with $q>p$.
We prove that every finite dimension
$N\leq\dim S_{p}^{(q)}$ satisfies the inequality
\begin{equation*}
	\left(\frac{2}{N}\right)^{1/2}
	\left(
	\frac{\Gamma\left(\frac{N+q'}{2}\right)}
	{\Gamma\left(\frac{N}{2}\right)}
	\right)^{1/q'}
	\leq K'_{p,q}(S)\gamma_{q'},
\end{equation*}
where the constant $K'_{p,q}(S)>0$ is the norm of the inclusion
of $S_{p}^{(q)}$ into $L_{q}(M,\,d\mu)$, measured against
$\|\cdot\|_{2,\mu}$, and $\gamma_{q'}$ is the
$L_{q'}$-norm of a standard Gaussian variable,
$1/q+1/q'=1$.
The left-hand side is at most $1$ and tends to $1$
as $N\to\infty$.
For $q>2$, the condition
$K'_{p,q}(S)\gamma_{q'}<1$
is sufficient for finite dimensionality, and then
\begin{equation*}
	\dim S_{p}^{(q)}
	\leq
	\frac{q'}
	{\left(K'_{p,q}(S)\gamma_{q'}\right)^{
			-\frac{2q}{q-2}}-1}.
\end{equation*}
In the endpoint case $q=\infty$ one recovers a
Grothendieck-type theorem: a closed subspace of
$L_{p}(M,\,d\mu)$ contained in $L_{\infty}(M,\,d\mu)$
has dimension at most
$K_{p,\infty}(S)^{\max\{2,p\}}$, where
$K_{p,\infty}(S)$ is the norm of the inclusion
into $L_{\infty}(M,\,d\mu)$ measured against
the $L_p$-norm.	
\end{abstract}

\keywords{closed Banach subspace, isometry, embedding, finite-dimensionality, probabilistic measure}


\maketitle

\section{Introduction}

The problem of estimating the dimension of closed linear subspaces of the
 Banach space $L_{p}(M,d\mu )$ for $p>1$ with $p\neq 2$ is
classical in Banach-space theory. Such estimates play an important role in
operator theory, approximation theory, and related areas \cite%
{BlPr,BuBe,KaPe,KrPeSe,LySh,MiGr,Nowa,Plic}. Applications also arise in
dynamical systems and other branches of analysis \cite%
{FoSeTe,Kato,Lady-1,Lady-2,Nino,Tao,Tema}. A well-known result in this
direction is the classical theorem of Grothendieck, which provides an
estimate for the dimension of a closed linear subspace inclusion $%
L_{p}(M,d\mu )\supset S_{p}^{(\infty )}\hookrightarrow L_{\infty }(M,d\mu ),$
and its generalization on the case of a linear closed subspace $%
S_{p}^{(c)}\subset C(M,\mathbb{R})$ of continuous functions, $\mathbb{\ \ \ }
$ embedded into $L_{2}(M,d\mu ).$ In this paper we consider the related
problem of estimation of the dimension of closed linear subspaces {\ }$%
S_{p}^{(q)}\subset L_{p}(M,d\mu )$ of the Banach space $%
L_{p}(M,d\mu )$ with respect to a probability measure $d\mu $ on $M,$ 
embedded into $L_{q}(M,d\mu ),$ where $q>p>1.$ Related estimates for $L_p$-subspaces and Bergman spaces appear in
\cite{BaPr-2,PrBa,LiWu}. An earlier preprint by the present authors
\cite{PrBa-preprint} considered a related estimate for the restricted
family of exponents $q=2+(p-2)2^m>p$. In contrast, the present
argument covers every $q>p>1$ with $q\geq2$, using the closed graph
theorem, interpolation, and spherical averaging to obtain the
normalized estimate of Theorem~\ref{Tm_1}. The Gaussian and Rademacher
discussion and the Grothendieck-type propositions are included as
comparisons with known results and as endpoint consequences.

The main result is stated in the following theorem.

\begin{theorem}
	\label{Tm_1} \ Let us consider a closed linear subspace $S_{p}^{(q)}\subset
	L_{p}(M,d\mu ),$ if $q>p>$ $1$ and $q\geq 2,$ where $d\mu $ is a
	probability measure on $M.$ Put also, by definition, the threshold constant $%
	\gamma _{q^{\prime }}=\sqrt{2}\left( \frac{\Gamma (\frac{q^{\prime }+1}{2})}{%
		\sqrt{\pi }}\right) ^{1/q^{\prime }},$ where $1/q+1/q^{\prime }=1.$ Then $\ $
	the following statements hold:
	
	i) the dimension $\dim S_{p}^{(q)}\geq N\in \mathbb{N}$ of the closed subspace 
	$S_{p}^{(q)}\subset L_{p}(M,d\mu )$ proves to satisfy the inequality 
	$\left( \frac{2}{N}\right) ^{1/2}\left( \frac{\Gamma (\frac{N+q^{\prime }}{2})%
	}{\Gamma (\frac{N}{2})}\right) ^{1/q^{\prime }}\leq 
	K_{p,q}^{\prime }(S)\gamma _{q^{\prime }}$, $1/q^{\prime }+1/q=1,$ for
	the bounded constant $K_{p,q}^{\prime }(S)>0,$ depending on the inverse
	domain inclusion mapping $\ J_{p}^{(q)}:L_{p}(M,d\mu )\supset S_{p}^{(q)}
	\hookrightarrow L_{q}(M,d\mu );$

ii) If $q>2$ and $K'_{p,q}(S)\gamma_{q'}<1$,
then $S_{p}^{(q)}$ is finite-dimensional
and $\dim S_{p}^{(q)}\leq
	\frac{q^{\prime}}{\left( K_{p,q}^{\prime}(S)\gamma_{q^{\prime }}\right)^{-\frac{2q}{q-2}}-1}$.
	If $K'_{p,q}(S)\gamma_{q'}\ge1$,
	this criterion gives no dimension bound.
\end{theorem}

\begin{remark}
	\ Taking into account the estimation of the dimension $\dim S_{p}^{(q)}=N\in 
	\mathbb{N}$ of a linear closed subspace $S_{p}^{(q)}\subset $ $L_{p}(M,d\mu
	),$ included into $L_{q}(M,d\mu )\ $\ for $q>p>1,$ $q\geq 2,$ obtained in
	Theorem \ref{Tm_1}, it is interesting to analyze its interpretation and
	possible relationship to known results from the book \cite{Pisi} by G.
	Pisier, equivalently formulated below. \ {}
\end{remark}

\begin{theorem}
    \label{Tm_1.0}
    \ In $L_2((0,1),d\lambda)$ there exists an infinite-dimensional
    closed linear subspace $\mathcal{G}$ generated by a sequence of
    independent standard Gaussian random variables. For every $1\leq p<\infty$,
    $\mathcal{G}\subset L_p((0,1),d\lambda)$ and every $f\in\mathcal{G}$
    satisfies $\|f\|_p=\gamma_p\|f\|_2$, where $\gamma_p$ is the
    $L_p$-norm of a standard Gaussian variable.
\end{theorem}

Moreover, the following interesting theorem about the Rademacher
stochastically independent random variables on the interval $[0,1]$ holds.

\begin{theorem}
	\label{Tm_1.1}\ The space $L_{p}((0,1),d\lambda ),1\leq p<\infty$, contains the
	linear closed subspace of the Rademacher orthonormal functions 
	$\mathcal{R}_{2}:=\overline{\mathrm{span}_{\mathbb{R}}}\{r_{n}(t)=$ $\mathrm{sign\ \sin }%
	(2^{n}\pi t):\ t\in \lbrack 0,1]$, $n\in \mathbb{N}\}\subset
	L_{2}((0,1),d\lambda )$, constituting a sequence of stochastically
	independent random variables, taking values $\pm 1$ with probability $1/2,$
	for which there exist positive constants $A_{q}(\mathcal{R})\leq B_{q}(%
	\mathcal{R}),1<q<\infty,$ such that for any $f\in $ $\mathcal{R}_{2}\subset
	L_{2}((0,1),d\lambda ),$ included into $L_{q}((0,1),d\lambda ),1<q<\infty,$ there
	hold the inequalities $A_{q}(\mathcal{R})||f||_{2}\leq ||f||_{q}$ $\leq
	B_{q}(\mathcal{R})||f||_{2}.$
\end{theorem}

To illustrate the dimension criterion of Theorem \ref{Tm_1}, recall that owing to
the Kadec-Pełczyński result \cite{KaPe}, the Rademacher closed
subspace $\mathcal{R}_{p},p>1,$ is complemented in $L_{p}((0,1),d\lambda ).$
In particular, if our closed subspace
$S_2^{(q)}\subset L_2((0,1),d\lambda)$,
embedded into $L_q((0,1),d\lambda)$,
$q>2$, contains the Rademacher subspace
$\mathcal{R}_2$, then $S_2^{(q)}$ is
infinite-dimensional. The dimension
criterion of Theorem \ref{Tm_1} is
consistent with this fact, as shown below.

\begin{remark}
    For this example, let us put $M=[0,1],$ $d\mu =d\lambda $ and consider the
	Rademacher closed subspace 
	$\mathcal{R}_{2}\subset $ $L_{2}((0,1),d\lambda), $ mentioned above. Take into account two inequalities: the first one,
	stated by Khintchine: 
	\begin{equation}
		 A_{q^{\prime }}(\mathcal{R})\leq ||\langle \xi |r\rangle ||_{q^{\prime
		}}\leq B_{q^{\prime }}(\mathcal{R})  \label{P0}
	\end{equation}%
for $\xi\in\mathbb{S}^{N-1}$,
$r=(r_1,r_2,\ldots,r_N)\in\mathcal{R}_2^N$,
$N\in\mathbb{N}$, where
$A_{q'}(\mathcal{R})<B_{q'}(\mathcal{R})$
are positive constants. For $q>2$, we have
$1<q'<2$, and the optimal upper constant is
$B_{q'}(\mathcal{R})=1$, while the optimal
lower constant is given by the Khintchine
inequality \cite{Haag},
and the second one (\ref{S2d}) at $\varphi =r$, $\mathcal{R}_{2}$ $\mathcal{%
		\subseteq }S_{2}^{(q)},$ as it is calculated in the next Section below:%
	\begin{equation}
		\ ||\langle \xi |r\rangle ||_{q^{\prime }}^{-1}\leq K_{2,q}^{\prime }(S).
		\label{P1}
	\end{equation}%
    Applying Theorem~\ref{Tm_1} to arbitrarily large finite
    Rademacher systems and letting $N\to\infty$ in its normalized
    left-hand side yields the inequality 
	\begin{equation}
		\sqrt{\frac{2}{N}}\ \left( \frac{\Gamma (\frac{N+q^{\prime }}{2})}{\Gamma (%
			\frac{N}{2})}\right) ^{1/q^{\prime }}\leq 1\leq \gamma _{q^{\prime
		}}K_{2,q}^{\prime }(S),  \label{P2}
	\end{equation}%
	which holds for all $q>2$, $1/q+1/q^{\prime }=1,$ and each natural integer $%
	N\in \mathbb{N}.$ 
	
    Since $\mathcal{R}_2$ is infinite-dimensional and
    $\mathcal{R}_2\subset S_2^{(q)}$, the subspace $S_2^{(q)}$ is
    infinite-dimensional by construction, not as a consequence of
    (\ref{P2}). Instead, (\ref{P2}) shows that the sufficient
    condition $\gamma_{q'}K'_{2,q}(S)<1$ does not hold in this
    example. Conversely, a closed subspace satisfying this strict
    condition is finite-dimensional by Theorem~\ref{Tm_1} and
    therefore cannot contain $\mathcal{R}_2$.
\end{remark}

In the special case, when $q=\infty ,$ as well as when closed linear
subspaces $S_{p}^{(c)}$ of the Banach space $L_{p}((0,1),d\mu )$, $p>1$,
consist of continuous functions, there are stated the following Grothendieck
type propositions.

\begin{proposition}
	\label{Pr_1}
	\ Let a linear closed topological subspace 
	$S_{p}^{(\infty)}\subset L_{p}(M,d\mu )$, $p>1,$ be embedded into a Banach
	space $L_{\infty }(M,d\mu ),$ where $d\mu $ is a probability measure on $M.$
	Then the dimension of the closed subspace
    $S_p^{(\infty)}$ satisfies
    $\dim S_p^{(\infty)}\leq K_{p,\infty}(S)^{\max\{2,p\}}$,
    where $K_{p,\infty}(S)$ is the norm of the inclusion into
    $L_\infty(M,d\mu)$ measured against the $L_p$-norm.
\end{proposition}

\begin{proposition}
	\label{Pr_2}
	\ Let $S_{p}^{(c)}$ be a closed subspace of the Banach space 
	$L_{p}((0,1),d\mu ),p>1,$ with respect to a full support probability measure $%
	d\mu $ on $[0,1],$ which allows the inverse domain inclusion into the Banach
	space $C([0,1];\mathbb{R})$ of  continuous functions on the interval $%
	[0,1]\subset \mathbb{R}_{+}.$ Then this subspace $S_{p}^{(c)}\subset
	L_{p}((0,1),d\mu )$ is finite-dimensional.
\end{proposition}

\section{Embedding of closed subspaces into Banach spaces}

Let us consider a closed linear subspace $S_{p}^{(q)}\subset L_{p}(M,d\mu )$
of the Banach space $L_{p}(M,d\mu ),$ $p>1,$ with respect to
a probability measure on $M,$ satisfying, in addition, the inclusion
condition $S_{p}^{(q)}\subset \left( L_{p}(M,d\mu );||\cdot ||_{p}\right) $
\ embedded into $\left( L_{q}(M,d\mu );||\cdot ||_{q}\right) \ $ for $%
q>p>1,q\geq 2.$ In order to state Theorem \ref{Tm_1} \ we need some two
lemmas.

\begin{lemma}
	\label{Lm_1} \ For any $q>p>1,$ $q\geq 2,$ \ there exists a bounded positive
	constant $K_{p,q}(S)\geq 1,$ such that 
	\begin{equation}
		||f||_{q}\leq K_{p,q}(S)\text{ }||f||_{p}  \label{S0}
	\end{equation}%
	for any $f\in S_{p}^{(q)}\ \hookrightarrow $ $L_{q}(M,d\mu ).$
\end{lemma}

\begin{proof}
	\ As the linear subspace $S_{p}^{(q)}\subset L_{p}(M,d\mu ),$ embedded into $%
	L_{q}(M,d\mu ),$ $q>p>1,$ $q\geq 2,$ is closed in $\ L_{p}(M,d\mu ),$ one
	can define the domain restricted inverse inclusion mapping 
	\begin{equation}
		J_{p}^{(q)}:L_{p}(M,d\mu )\supset S_{p}^{(q)}\hookrightarrow L_{q}(M,d\mu ).
		\label{1a}
	\end{equation}%
	If a sequence $\{f_{n}:n\in \mathbb{N}\}$ $\subset S_{p}^{(q)}$ converges in 
	$S_{p}^{(q)}\subset L_{p}(M,d\mu )$ to an element $f$ $\in $ $%
	S_{p}^{(q)}\subset L_{p}(M,d\mu )$ with respect to the norm on $\
	L_{p}(M,d\mu )$ and simultaneously its image $\{J_{p}^{(q)}f_{n}:n\in 
	\mathbb{N}\}$ $\subset S_{p}^{(q)}\hookrightarrow L_{q}(M,d\mu )$ converges
	to an element $g\in L_{q}(M,d\mu )\subset L_{p}(M,d\mu ),$ $q>p>1,$ $q\geq
	2, $ with respect to the norm on$\ L_{q}(M,d\mu ),$ one can identify these
	limiting functions $f\sim g$ \ almost everywhere. Really, since $%
	L_{q}(M,d\mu )\subset L_{p}(M,d\mu )$ for every $q>p\ >1,$ from the
	estimations 
	\begin{equation}
		\begin{array}{c}
			||\ f-g||_{p}\leq ||f-f_{n}||_{p}+||g-f_{n}||_{p}\leq \\ 
			\leq ||f-f_{n}||_{p}+||\left( g-f_{n}\right) \cdot 1||_{p}\leq \\ 
			\leq ||f-f_{n}||_{p}+||\left( g-J_{p}^{(q)}f_{n}\right) ||_{q}\ \overset{%
				n\rightarrow \infty }{\rightarrow }0%
		\end{array}
		\label{1b}
	\end{equation}%
	one obtains that $f\sim g\in S_{p}^{(q)}\hookrightarrow L_{q}(M,d\mu )$
	almost everywhere 
and that the graph of
$J_p^{(q)}: (S_p^{(q)},\|\cdot\|_p)\to L_q(M,d\mu)$ is closed.
Because $S_p^{(q)}$ is Banach in the $L_p$ norm,
	the Banach closed graph theorem \cite{Bana,BaLyMy,Lax,ReSi,Rudi} gives
	rise to the existence of such a positive constant $K_{p,q}(S)<\infty $ that 
	\begin{equation}
		||\ f||_{q}\leq K_{p,q}(S)\text{ }||f||_{p}  \label{2}
	\end{equation}%
	for arbitrary $f\in S_{p}^{(q)}\subset L_{p}(M,d\mu )$, embedded into 
	$L_{q}(M,d\mu )$, $q>p>1$, $q\geq 2.$ Remark also, that the following estimations 
	\begin{equation}
		||f||_{2}\leq ||\ f||_{q}=||J_{p}^{(q)}f||_{q}\leq K_{p,q}(S)\text{ }%
		||f||_{p}<\infty  \label{3}
	\end{equation}%
	hold for any 
	$f\in S_{p}^{(q)}\hookrightarrow L_{q}(M,d\mu)$, $q>2$, $q>p>1,$ 
	easily following from the H\"{o}lder inequality. $\Box $
\end{proof}

Taking into account Lemma \ref{Lm_1}, we can formulate the next lemma, which
is in some sense the converse to the inequality (\ref{S0}).

\begin{lemma}
	\label{Lm_2}
	\ There exists a constant $K'_{p,q}(S)>0$
	such that the following inequality
	\begin{equation}
		\|f\|_q\leq K'_{p,q}(S)\|f\|_2
		\label{S1}
	\end{equation}
	holds for any $f\in S_p^{(q)}$.
	If $K'_{p,q}(S)$ is chosen as the norm
	of this inclusion, then
	\begin{equation}
		K'_{p,q}(S)\leq
		\begin{cases}
			K_{p,q}(S), & 1<p\leq2,\quad q>p,\\[2mm]
			K_{p,q}(S)^{\frac{(q-2)p}{2(q-p)}},
			& 2<p<q.
		\end{cases}
		\label{S1b}
	\end{equation}
\end{lemma}

\begin{proof}
	\ If $1<p\leq 2,$ from the H\"{o}lder inequality 
	\begin{equation}
		||f||_{p}\leq \ ||f||_{2}  \label{S1a}
	\end{equation}%
	for any $f\in S_{p}^{(q)}\ \hookrightarrow L_{q}(M,d\mu )$ one obtains 
	inequality (\ref{S1}) for the bounded constant $K_{p,q}^{\prime }(S)\leq%
	K_{p,q}(S)>0.$ If $p>2,$ then one can make use of the following
	interpolation inequality:\ \ \ 
	\begin{equation}
		||f||_{p}\leq ||f||_{q}^{\theta }\text{ }||f||_{2}^{1-\theta },  \label{S2}
	\end{equation}%
	which holds for any $f\in S_{p}^{(q)}\subset L_{p}(M,d\mu ),$ where $%
	1/p=\theta /q+(1-\theta )/2\ $and $\theta =\frac{q\,(p-2)}{(q-2)p}\in (0,1),$
	$1-\theta =\frac{2\,(q-p)}{(q-2)p}\in (0,1).$ Now having put, by
	definition, $q>p>2$, the inequality (\ref{S2}) \ jointly with that of \ (%
	\ref{S0}) gives rise to the searched estimation \ (\ref{S1}), where the
	bounded constant $K_{p,q}^{\prime }(S)\leq K_{p,q}(S)^{\frac{(q-2)p}{2(q-p)}}$
	$>0,$ thus proving the lemma. \bigskip\ $\Box $
\end{proof}

\bigskip

\begin{proof}
	\ (\textit{Proof of Theorem \ref{Tm_1}).} Based on the lemmas above, one can
	proceed to proving our theorem.  

By H\"older's inequality and
Lemma~\ref{Lm_2}, for every
$u\in S_p^{(q)}$ we obtain
\begin{equation}
	\|u\|_2^2
	\leq\|u\|_q\|u\|_{q'}
	\leq K'_{p,q}(S)\|u\|_2\|u\|_{q'}.
	\label{S2b_b}
\end{equation}
Consequently,
\begin{equation}
	\|u\|_2\leq K'_{p,q}(S)	
	\label{S2b}
\end{equation}
if $\|u\|_{q'}=1$.
	Choose an orthonormal family $\Phi =\{\varphi _{1},\varphi
	_{2},...,\varphi _{N}\}\subset (S_{p}^{(q)};||\cdot ||_{2})$ for some $%
	N\leq \dim S_{p}^{(q)}$ $\in \mathbb{N},$ $(\varphi _{j}|\varphi _{k})=$ $%
	\int_{M}$ $\varphi _{j}\varphi _{k}d\mu =$ $\delta _{jk},||\varphi
	_{j}||_{2}=1,j,k=\overline{1,N},$ one can observe that a function $\varphi
	_{a}:=$ $\langle a|\varphi \rangle _{N}$ $=$ $\sum_{j=1}^{N}a_{j}\varphi
	_{j} $ $\in $ $(S_{p}^{(q)};||\cdot ||_{2})$ has the norm 
	\begin{equation}
		||\varphi _{a}||_{2}=\left( \sum_{j=1}^{N}|a_{j}|^{2}\right) ^{1/2}=|a|_{N},
		\label{S2bb}
	\end{equation}%
	where the vector $\varphi :=(\varphi _{1},\varphi _{2},...,\varphi
	_{N})^{\intercal }\in \left( S_{p}^{(q)}\right) ^{N}$ and took 
	$a\in E_N,$ as an arbitrary vector. Having substituted the value of the norm
	(\ref{S2bb}) into (\ref{S2b}), one obtains the inequality 
	\begin{equation}
		|a|_{N}\text{ }\leq K_{p,q}^{\prime }(S),  \label{S3}
	\end{equation}%
	which should be combined with the imposed above condition $||\varphi
	_{a}||_{q^{\prime }}=1.$ Taking into account that 
	\begin{equation}
		||\varphi _{a}||_{q^{\prime }}=||\langle a|\varphi \rangle _{N}||_{q^{\prime
		}}=|a|_{N}||\langle \xi |\varphi \rangle _{N}||_{q^{\prime }}\text{ }=1,
		\label{S2c}
	\end{equation}%
	where $a\in E_N\ $\ and $\xi :=a/|a|_{N}$ $\in \mathbb{S}%
	^{N-1},|\xi |_{N}=1,\ $ yielding jointly with \ (\ref{S3}) the inequality 
	\begin{equation}
		||\langle \xi |\varphi \rangle _{N}||_{q^{\prime }}^{-1}\leq K_{p,q}^{\prime
		}(S),  \label{S2d}
	\end{equation}%
We can get rid of the spherical variables
$\xi\in\mathbb{S}^{N-1}$ by applying the
averaging method \cite{GeSh,John,Vile}
to inequality (\ref{S2d}).
Namely, putting $r=q'$ and integrating
over the unit sphere with respect to the
normalized spherical measure $d\sigma(\xi)$,
we obtain
\begin{equation}
	\begin{aligned}
		(K_{p,q}'(S))^{-r}
		&\leq
		\int_{\mathbb{S}^{N-1}}d\sigma(\xi)
		\int_M|\langle\xi,\varphi\rangle_N|^r\,d\mu\\
		&=
		\frac{\Gamma(\frac{N}{2})
			\Gamma(\frac{r+1}{2})}
		{\sqrt{\pi}\Gamma(\frac{N+r}{2})}
		\int_M|\varphi|_N^r\,d\mu .
	\end{aligned}
	\label{S3a}
\end{equation}
Since $1<r\leq 2$, by Jensen's inequality,
\begin{equation}
	\begin{aligned}
		\int_M|\varphi|_N^r\,d\mu
		&\leq
		\left(\int_M|\varphi|_N^2\,d\mu
		\right)^{r/2}=N^{r/2}.
	\end{aligned}
	\label{A0}
\end{equation}
Combining (\ref{S3a}) and (\ref{A0})
yields the final numerical estimation	
\begin{equation}
		\left( \frac{2}{N}\right) ^{1/2}\left( \frac{\Gamma (\frac{N+q^{\prime }%
			}{2})}{\Gamma (\frac{N}{2})}\right) ^{1/q^{\prime }}\leq
		K_{p,q}^{\prime }(S)\gamma _{q^{\prime }},  \label{A0a}
\end{equation}
where we put, by definition, the threshold constant $\gamma _{q^{\prime }}=%
	\sqrt{2}\left( \frac{\Gamma (\frac{q^{\prime }+1}{2})}{\sqrt{\pi }}\right)
	^{1/q^{\prime }}$, $1/q+1/q^{\prime }=1.$ Moreover, if the right-hand side is strictly less than $1$,
this ensures that $S_p^{(q)}$ is finite-dimensional for $q>2$.
If $K'_{p,q}(S)\gamma_{q'}\geq1$, this criterion yields no
finite-dimensionality conclusion.
	
	Taking into account the inequality \cite{wendel} $\left( \frac{2}{N}\right) ^{q^{\prime
		}/2}\frac{\Gamma (\frac{N+q^{\prime }}{2})}{\Gamma (\frac{N}{2})}\geq(1+%
	\frac{q^{\prime }}{N})^{-\frac{2-q^{\prime }}{2}}$ for all 
	$N\in\mathbb{N},$ from (\ref{A0a}) one obtains easily the following estimation
	of the dimension $\mathrm{dim}$ $S_{p}^{(q)}:$ 
	\begin{equation}
		\ \mathrm{dim}\text{ }S_{p}^{(q)}\leq \mathrm{\ }\frac{q^{\prime }}{\left(
			K_{p,q}^{\prime }(S)\gamma _{q^{\prime }}\right) ^{-\frac{2q}{q-2}%
			}-1}.  \label{B0}
	\end{equation} $\Box $
	\end{proof}

\begin{proof}	
	\ (\textit{Proof of Proposition \ref{Pr_1}).}
In the endpoint case $q=\infty$,
we use bounded point-evaluation
functionals on finite-dimensional
subspaces to obtain the corresponding
dimension estimate.
	By the closed graph theorem, there exists a constant
	$K_{p,\infty}(S)>0$ such that
	$\|f\|_\infty\leq K_{p,\infty}(S)\|f\|_p$
	for any $f\in S_p^{(\infty)}$.
	Using the inequalities between the $L_p$-, $L_2$-
	and $L_\infty$-norms, we obtain
\begin{equation*}
	\|f\|_\infty\leq K'_{p,\infty}(S)\|f\|_2,
	\qquad
	K'_{p,\infty}(S)
	\leq K_{p,\infty}(S)^{\max\{1,p/2\}},
\end{equation*}
	where $K'_{p,\infty}(S)$ is the inclusion norm
	 measured against $\|\cdot\|_2$.
	
	Let $\Phi_N=\{\varphi_1,\ldots,\varphi_N\}
	\subset S_p^{(\infty)}$ be an arbitrary finite
	$L_2$-orthonormal family, and put
	$E_N=\mathrm{span}\Phi_N$.
	Choosing measurable representatives and applying
	the preceding estimate to all rational linear
	combinations of $\varphi_1,\ldots,\varphi_N$,
	one can obtain a common set $M_N\subset M$ of full
	measure on which
\begin{equation*}
	\left|\sum_{j=1}^N a_j\varphi_j(x)\right|
	\leq K'_{p,\infty}(S)
	\left(\sum_{j=1}^N|a_j|^2\right)^{1/2}
\end{equation*}
	for all $a\in\mathbb R^N$ and $x\in M_N$.
	As a result, the point evaluations
	$l_x(f):=f(x)$ are well-defined bounded
	linear functionals on $E_N$ for $x\in M_N$,
	\begin{equation}
		|l_{x}(f)|\text{ }\leq K_{p,\infty }^{\prime }(S)||f||_{2}  \label{B1}
	\end{equation}%
	for any $f\in E_N\subset L_{\infty}(M;d\mu)$ and $x\in M_N$. 
Having calculated the norm of this
functional, we obtain
\begin{equation}
	\sup_{\substack{f\in E_N\\f\neq0}}
	\frac{|l_x(f)|}{\|f\|_2}
	=\|l_x\|
	\leq K'_{p,\infty}(S).
	\label{B2}
\end{equation}
By the Riesz theorem,
for every $x\in M_N$ there exists
$g_x\in E_N$ such that
$l_x(f)=(g_x|f)$ and
$\|l_x\|=\|g_x\|_2$
for all $f\in E_N$.
Since $\Phi_N=\{\varphi_1,\ldots,\varphi_N\}$
is an orthonormal basis of $E_N$,
the related Parseval equality is
\begin{equation}
	\|g_x\|_2^2
	=\sum_{j=1}^{N}|(g_x|\varphi_j)|^2
	=\sum_{j=1}^{N}|\varphi_j(x)|^2.
	\label{B3}
\end{equation}
	Combining this with inequality (\ref{B2}) gives the following inequality 
\begin{equation}
	\sum_{j=1}^{N}|\varphi_j(x)|^2
	\leq (K'_{p,\infty}(S))^2,
	\label{B4}
\end{equation}
	which holds for any $x\in M_N$. Having integrated the obtained inequality 
	(\ref{B4}) over the whole space $M,$ we obtain that 
\begin{equation}
	N\leq (K'_{p,\infty}(S))^2
	\leq K_{p,\infty}(S)^{\max\{2,p\}}.
	\label{B5}
\end{equation}
Since the orthonormal family $\Phi_N$
was arbitrary, the obtained estimate holds
for every finite orthonormal family in
$S_p^{(\infty)}$. Consequently,
\[
\dim S_p^{(\infty)}
\leq K_{p,\infty}(S)^{\max\{2,p\}}<\infty,
\]
thus proving Proposition \ref{Pr_1}.
$\Box$
\end{proof}

Moreover, as a technical consequence of the
results above, regarding the Grothendieck type \cite{Grot} inclusion
problem of a closed subspace $S_{p}^{(c)}\subset L_{p}((0,1),d\mu ),p>1,$
into the Banach space $C([0,1];\mathbb{R})$ with respect to a full support
probability measure $d\mu $ on $[0,1],$ we present below a proof of
Proposition \ref{Pr_2}.

\begin{proof}
	\ (\textit{Proof of Proposition \ref{Pr_2}}). As a closed subspace 
	$L_{p}((0,1),d\mu )\supset S_{p}^{(c)}\hookrightarrow C([0,1];\mathbb{R})$ of
	functions on the interval $[0,1]$ is continuously embedded into
	the Banach space $C([0,1],\mathbb{R}),$ from the Banach closed mapping
	theorem  \cite{Lax,Rudi} one derives the existence of such a constant $%
	K_{p}>1$ that the corresponding inverse inclusion mapping $%
	J_{p}^{(c)}:L_{p}(0,1;d\mu )\supset S_{p}^{(c)}\hookrightarrow C([0,1],%
	\mathbb{R})$ is bounded, that is 
	\begin{equation}
		||J_{p}^{(c)}f||_{\infty }\leq K_{p,c}(S)||f||_{p}  \label{S14}
	\end{equation}%
	for any $f\in S_{p}^{(c)}\hookrightarrow (C([0,1];\mathbb{R}),||\cdot
	||_{\infty }).$ Remark now that if $1<p\leq 2,$ then the inequality $%
	||f||_{p}\leq ||f||_{2}$ holds for all $f\in S_{p}^{(c)}\subset C([0,1];%
	\mathbb{R})\subset L_{p}((0,1),d\mu ).$ If $p>2,$ we can observe that $%
	|f|^{p}\leq ||f||_{\infty }^{p-2}|f|^{2}$ for any $f\in S_{p}^{(c)}\subset
	C([0,1];\mathbb{R})\subset L_{p}((0,1),d\mu ),$ whence by integration over
	the interval $[0,1]$ one easily obtains that $||f||_{p}\leq ||f||_{\infty }^{%
	\frac{p-2}{p}}||f||_{2}^{\frac{2}{p}}.$ Substituting the latter inequality
	into  (\ref{S14}),  we obtain that $\ ||f||_{\infty
	}\leq K_{p,c}(S)^{p/2}||f||_{2},$ what jointly with the evident inequality $%
	||f||_{2}\leq ||f||_{\infty }$ gives rise to the dual inequality 
	\begin{equation}
	\|f\|_2\le\|f\|_\infty\le
	K_{p,c}(S)^{\max\{1,p/2\}}\|f\|_2.  \label{S15}
	\end{equation}%
	for all $f\in S_{p}^{(c)}\subset C([0,1];\mathbb{R})$. 
	The inequalities above also show that
	the $L_p$- and $L_2$-norms are equivalent
	on $S_p^{(c)}$. Hence this subspace is
	complete with respect to the $L_2$-norm.
    Because $\mu$ has full support, two continuous functions that
    agree $\mu$-almost everywhere agree everywhere on $[0,1]$;
    hence their point evaluations are well-defined.
	 As above, define
	for any $t\in \lbrack 0,1]$ a bounded linear functional $l_{t}:$ $%
	(S_{p}^{(c )};||\cdot ||_{2})$ $\rightarrow \mathbb{R}$ on the Hilbert
	subspace $(S_{p}^{(c)};||\cdot ||_{2}),$ such that $l_{t}(f)=f(t)\in 
	\mathbb{R},$ which allows owing to the Riesz theorem the representation $%
	l_{t}(f)=(g_{t}|f)$ for all $f\in (S_{p}^{(c)};||\cdot ||_{2}),$ where 
	$g_{t}\in $ $(S_{p}^{(c)};||\cdot ||_{2})$ and  $||l_{t}||
	=||g_{t}||_{2}.$ 

Let
$\Phi_N^{(c)}=\{\varphi_1,\ldots,\varphi_N\}
\subset S_p^{(c)}$
be an arbitrary finite $L_2$-orthonormal
family, that is,
$\|\varphi_j\|_2=1$ and
$(\varphi_j|\varphi_k)=\delta_{jk}$,
$j,k=1,\ldots,N$.
By Bessel's inequality,
\begin{equation}
	\sum_{j=1}^{N}|\varphi_j(t)|^2
	=\sum_{j=1}^{N}|(g_t|\varphi_j)|^2
	\leq\|g_t\|_2^2.
	\label{S15aa}
\end{equation}
Combining this with inequality (\ref{S15})
gives
\begin{equation}
	\sum_{j=1}^{N}|\varphi_j(t)|^2
	\leq K_{p,c}(S)^{\max\{2,p\}},
	\label{S15ab}
\end{equation}
which holds for every $t\in[0,1]$.
Integration over $[0,1]$ yields
\begin{equation}
	N\leq K_{p,c}(S)^{\max\{2,p\}}.
	\label{S15bb}
\end{equation}
Since $N$ was arbitrary, we obtain
\[
\dim S_p^{(c)}
\leq K_{p,c}(S)^{\max\{2,p\}}<\infty,
\]
thus proving Proposition \ref{Pr_2}.
$\Box$
\end{proof}

\section{Conclusion}

We have considered closed linear subspaces $S_{p}^{(q)}$ of the  
Banach space $(L_{p}(M,d\mu );||\cdot ||_{p})$, $p>1$, allowing the
embedding into the Banach space $(L_{q}(M,d\mu );||\cdot ||_{q})$, $q>p>1$,
$q\geq 2,$ regarding a probability measure $d\mu $ on $M.$\ We derived the
numerical estimation $\ \left( \frac{2}{N}\right) ^{1/2}\left( \frac{%
\Gamma (\frac{N+q^{\prime }}{2})}{\Gamma (\frac{N}{2})}\right)
^{1/q^{\prime }}\leq K_{p,q}^{\prime }(S)\gamma _{q^{\prime }},$ where $%
1/q^{\prime }+1/q=1,$ on the dimension $\dim S_{p}^{(q)}\geq N\in \mathbb{N}$
of a closed embedded subspaces $S_{p}^{(q)}\subset
(L_{p}(M,d\mu );||\cdot ||_{p}),$ embedded into $(L_{q}(M,d\mu );||\cdot
||_{q}),$ if $q\ >$ $p>1.$ Moreover, if $q>2$ and \ $K_{p,q}^{\prime }(S)\gamma
_{q^{\prime }}<1,$ then the following estimation of the dimension $\dim
S_{p}^{(q)}$ holds: 
\[
\mathrm{dim}\ S_{p}^{(q)}\leq \mathrm{\ }\frac{%
	q^{\prime }}{\left( K_{p,q}^{\prime }(S)\gamma _{q^{\prime }}\right)
^{-\frac{2q}{q-2}}-1}.
\]
 In case of the space $M=[0,1]\subset \mathbb{R}%
_{+}, $ endowed with an arbitrary full support probability measure $d\mu ,$
we stated the Grothendieck type finite-dimensionality result for a linear
closed subspace $S_{p}^{(c)}\subset $ $\left( L_{p}((0,1),d\mu );||\cdot
||_{p}\right) ,$ embedded into the Banach space $\left( C([0,1],\mathbb{R}%
);||\cdot ||_{\infty }\right) .$ When $q>2$ and $K'_{p,q}(S)\gamma_{q'}\geq1$,
Theorem~\ref{Tm_1} gives no finite-dimensionality conclusion.
Infinite-dimensional examples occur in this regime, as illustrated
by the Rademacher subspace in Theorem~\ref{Tm_1.1}.
Further dimension bounds therefore require additional hypotheses
on the subspace or its embedding.

\section{Acknowledgements}

The authors thank the participants of the seminar at the Department of
Applied Mathematics, University of Agriculture in Kraków, for their
remarks and suggestions. Special thanks go to our colleagues M. Ptak
and A. K. Prykarpatski for their continued interest, fruitful discussions,
and help in preparing the manuscript.

\section*{Declarations}

\begin{itemize}
	\item \textbf{Funding:} This research was financed from the subsidy of the Ministry of Science and Higher Education for the Hugo Kołłątaj Agricultural University in Kraków for the year 2026.
	\item \textbf{Conflict of interest/Competing interests:} The authors declare that they have no known competing financial interests or personal relationships that could have appeared to influence the work reported in this paper.
	\item \textbf{Ethics approval and consent to participate:} Not applicable.
	\item \textbf{Consent for publication:} Not applicable.
	\item \textbf{Data availability:} Not applicable.
	\item \textbf{Materials availability:} Not applicable.
	\item \textbf{Code availability:} Not applicable.
	\item \textbf{Author contribution:} The authors contributed equally to this work.
\end{itemize}


\begin{thebibliography}{99}
\bibitem{BlPr} {\normalsize Prykarpatsky A.K., Blackmore D., A solution set
analysis of a nonlinear operator equation using a Leray--Schauder type
fixed point approach. Topology 48, (2009), p. 182--185}

\bibitem{AlOd} {\normalsize Alspach D.E., Odell E.W., Functional Analysis:
Proceedings of the Seminar at the University of Texas, Editors: E.W. Odell
and P. Rosenthal, Austin, TX, USA, 1986-87 }

\bibitem{Alze} {\normalsize Alzer H., Inequalities for the Gamma-function.
Proceedings of the AMS, 128 (2000), 141-147}

\bibitem{AsMa-1} {\normalsize Astashkin S.V. and Maligranda L., $%
L_{p}+L_{\infty }$ and $L_{p}\cap L_{\infty }$ are not isomorphic for all $%
1\leq p<\infty ,p\neq 2.$ Proc. Amer. Math. Soc. 146 (2018), 2181-2194 }

\bibitem{AsMa-2} {\normalsize Astashkin S.V. and Maligranda L., Rademacher
functions in Ces\`aro type spaces, Studia Mathematica, 198(3) (2010), 235-247}

\bibitem{BaPr-2} {\normalsize Balinsky A.A., Prykarpatski A.K.,
On the finite dimensionality of closed subspaces in
$L_p(M,d\mu)\cap L_q(M,d\mu)$. Axioms 10 (2021), article 275.}

\bibitem{Bana} {\normalsize Banach S., Theorie des operations lineaires,
PWN, Warszawa, 1932. }

\bibitem{BaBaPlPr} {\normalsize Banakh I., Banakh T., Plichko A.,
Prykarpatsky\ A.K., On local convexity of nonlinear mappings between Banach
spaces. Cent. Eur. J. Math., 10(6), (2012), p. 2264-2271 }

\bibitem{BaLyMy} {\normalsize Banakh T., Lyantse W.E. , Mykytyuk Ya.V., $%
\infty $-Convex sets and their applications to the proof of certain
classical theorems of functional analysis. Matematychni Studii. 11(1),
(1999), p. 83-84 }

\bibitem{BuBe} {\normalsize Butzer P.L., \ Berens H., Approximation,
Springer-Verlag Berlin Heidelberg, 1967 }

\bibitem{FoSeTe} {\normalsize Foias C., Sell G.R., Temam R. Inertial
manifolds for nonlinear evolution equations. J. Diff. Eqns., 73, (1988), p.
309-353 }

\bibitem{GeSh} {\normalsize Gel'fand I.M., Shilov G.E., Generalized
functions, Vol. 1: Properties and operations, Academic Press, New York, 1964.%
}

\bibitem{Grot} {\normalsize Grothendieck A., Sur certains sous-espaces
vectoriels de $L_{p}.$ Canadian J. Math. 6 (1954), p. 158--160 }

\bibitem{John} {\normalsize John F., Plane waves and spherical means.
Interscience Publisher, London, 1955}

\bibitem{JoSh} {\normalsize Johnson W.B. , Schechtman G., Finite dimensional
subspaces of $L_{p}.$ Elsevier Preprint, January 31, 2001, p. 1-40 }

\bibitem{Haag} {\normalsize Haagerup U.,  The best constants in Khintchine
inequality, Studia Mathematica, 70 (1981), p. 231-283}

\bibitem{KaPe} {\normalsize Kadec M.I, Pe\l czy\'{n}ski A., Bases, lacunary
sequences and complemented subspaces in the spaces }$L_{p}${\normalsize ,
Studia Math., 21 (1962), 161-176}

\bibitem{Kato} {\normalsize Kato T. Nonlinear evolution equations in Banach
spaces, Proceed. Symp. On Pure Appl. Math., 45(2), (1986), p. 9-23 }

\bibitem{KrPeSe} {\normalsize Kre\u{\i}n S.G. , Petun\={\i}n Yu.I, and Sem%
\"{e}nov E.M., Interpolation of linear operators, Translations of
Mathematical Monographs, vol. 54, American Mathematical Society, Providence,
R.I., 1982. }

\bibitem{Lady-1} {\normalsize Ladyzhenskaya O.A., Finite-dimensionality of bounded invariant sets for Navier-stokes systems and other dissipative systems. J Math Sci 28, 714-726 (1985) }

\bibitem{Lady-2} {\normalsize O. A. Ladyzhenskaya, Estimates of the fractal dimension and of the number of deterministic modes for invariant sets of dynamical systems, in Zap. Nauchn. Semin. Leningr. Otd. Mat. Inst. Steklova Vol. 163: Boundary-Value Problems of Mathematical Physics and Related Questions of Function Theory, no. 19 (Nauka, Leningrad, 1987), pp. 105–129 }

\bibitem{Lax} {\normalsize Lax P.D., Functional analysis, Wiley-Interscience
2002.}

\bibitem{Lewi} {\normalsize Lewis D.R., Ellipsoids defined by Banach ideal
norms, Mathematika, 26(1), (1979), p. 18-29 }

\bibitem{LiLo} {\normalsize Lieb E.H., Loss M., Analysis, Graduate Studies
in Mathematics, v. 14, American Mathematical Society, Providence, R.I., 1997 
}

\bibitem{LiWu} {\normalsize Liu Y., Wu J., Xiong Y., The Grothendieck theorem in
Bergman spaces. arXiv:2511.01521v1 [math.CV] 3 Nov 2025 }

\bibitem{LySh} {\normalsize Lyubich Yu.I., Shatalova O.A., Isometric
embeddings of finite-dimensional }$l_{p}${\normalsize -spaces over the
quaternions, Algebra and Analysis, 16(1) (2005), 9--24}

\bibitem{MiGr} {\normalsize Michalak A., Grala-Michalak J.,On constructions
of isometric embeddings of nonseparable }$L_{p}${\normalsize \ spaces, }$%
0<p\leq 2.${\normalsize \ Commentationes Mathematicae, 48(2) (2008), 129-145}

\bibitem{Nino} {\normalsize Ninomiya H. Some remarks on inertial manifolds.
J Math.Kyoto Univ., 32(4) (1992),  667-688 }

\bibitem{Nowa} {\normalsize Nowak P.W., On coarse embeddibility into }$l_{p}$%
{\normalsize \ spaces \ and conjecture of Dranishnikov, arXiv:math/0410566v1
[math.MG] 27 Oct 2004}

\bibitem{Pisi} {\normalsize Pisier G., The Volume of Convex Bodies and
Banach Space Geometry, Cambridge Tracts in Mathematics, 94 (1989)}

\bibitem{Plic} {\normalsize Plichko A., Rate of Decay of the Bernstein
Numbers, Journal of Mathematical Physics, Analysis, Geometry} {\normalsize %
9(1) 2012, 1-14}

\bibitem{PrBa} {\normalsize Prykarpatski A.K., Balinsky A.A., Some remarks
on linear closed subspaces in Bergman spaces, Bull. L.N. Gumilyov Eurasian
Nat. Univ., Math. Comput. Sci. Mech. Ser. 154(1) (2026), 6--14.
DOI: 10.32523/bulmathenu.2026/1.1.}

\bibitem{PrBa-preprint} {\normalsize Prykarpatskyy Y.A., Balinsky A.A.,
On closed linear subspaces embedded into functional Banach spaces and their
finite-dimensionality, arXiv:2606.19608 (2026).}


\bibitem{ReSi} {\normalsize Reed M. and Simon B., Methods of modern
mathematical physics, vol. 1. Functional analysis, Academic Press, New York
1972.}

\bibitem{RoSe} {\normalsize Rodin V.A. and Semyonov E.M., Rademacher
series in symmetric spaces, Anal. Math. 1 (1975), 207--222}

\bibitem{Rudi} {\normalsize Rudin W., Functional Analysis, McGraw-Hill, 1991.%
}

\bibitem{Tao} {\normalsize Tao T., An introduction to measure theory,
Graduate Studies in Mathematics, v. 126, American Mathematical Society,
Providence, R.I., 2011 }

\bibitem{Tema} {\normalsize Temam R., Infinite-dimensional dynamical systems
in fluid mechanics, Proceed. Symp. On Pure Appl. Math., 45(2), (1986), p.
431-445 }

\bibitem{Vile} {\normalsize Vilenkin N.Ja., Special functions and the theory
of group representations, Transl. Math. Monogr. 22, Amer. Math. Soc.,
Providence, RI, 1968.}

\bibitem{wendel} {\normalsize  Wendel J.G., Note on the gamma function, 
The American Mathematical Monthly, (1948), Vol. 55, No. 9, pp. 563–564.}

\end{thebibliography}
\end{document}